\documentclass[11pt,a4paper]{amsart}
\usepackage{subfiles}
\usepackage{leading}
 \pdfoutput=1
 \usepackage{blindtext}
 \usepackage{amsaddr}
\usepackage[utf8]{inputenc}
\usepackage{graphicx,amssymb,amstext,amsmath,mathtools}
\usepackage{tikz,caption,subcaption,amsthm}
\usepackage{enumitem,dsfont,textcomp}
\usepackage[all,cmtip]{xy}
\usepackage[backref=page]{hyperref}
\usepackage{csquotes}
\usepackage{MnSymbol}
\usepackage{tikz}
\usepackage{mathtools}
\usepackage{tikz-cd}

\newcommand{\xqedhere}[2]{%
  \rlap{\hbox to#1{\hfil\llap{\ensuremath{#2}}}}}
\makeatletter
\newcommand{\pushright}[1]{\ifmeasuring@#1\else\omit\hfill$\displaystyle#1$\fi\ignorespaces}
\newcommand{\pushleft}[1]{\ifmeasuring@#1\else\omit$\displaystyle#1$\hfill\fi\ignorespaces}
\makeatother

\theoremstyle{plain}
\newtheorem{theorem}{Theorem}[section]
\newtheorem{lemma}[theorem]{Lemma}
\newtheorem{proposition}[theorem]{Proposition}
\newtheorem{corollary}[theorem]{Corollary}
\theoremstyle{definition}

\numberwithin{equation}{section}
\renewcommand*{\backref}[1]{}
\renewcommand*{\backrefalt}[4]{%
   \ifcase #1 (Not cited.)%
   \or        (Cited on page~#2.)%
   \else      (Cited on pages~#2.)%
   \fi}

\begin{document}
\leading{14pt}
\author{Khadijeh Alibabaei}

\address{Centro de Matem\'{a}tica e Departamento de Matem\'{a}tica, Faculdade de Ci\^{e}ncias,
Universidade do Porto,
Porto, Portugal}
{\email{f.alibabaee@gmail.com}}

\title{The  pseudovariety of all nilpotent groups is tame}

\begin{abstract}
It has been shown that for every prime number $p$, the pseudovariety $\mathsf{G}_p$ of all finite $p$-groups is tame with respect to an implicit signature containing the canonical implicit signature.
In this paper we generalize this result and we  show that the pseudovariety of all finite nilpotent groups is tame but it is not completely tame.
\end{abstract}

\maketitle

\section{Introduction}
 By a \emph{pseudovariety} we mean a class of semigroups which is closed under taking subsemigroups, finite direct  products, and homomorphic images. A pseudovariety is said to be \emph{decidable} if there is an algorithm to test membership of a finite semigroup;   otherwise, the pseudovariety is said to be \emph{undecidable}.  Eilenberg \cite{Eilenberg:1974} established a correspondence between varieties of
rational languages and pseudovarieties of finite semigroups which translates problems in language theory into the decidability of pseudovarieties of semigroups.
In general the decidability of pseudovarieties is not preserved by many operations on pseudovarieties such as semidirect product, join and Mal'cev product (\cite{Rhodes:1997c, Albert&Baldinger&Rhodes:1992}).
 Almeida and Steinberg introduced a  refined version of decidability called  \emph{tameness} \cite{Jorge&Ben:2000}. The tameness property requires the \emph{reducibility } property which is  a generalization of the notion of inevitability that Ash introduced to prove the type II conjecture of Rhodes \cite{Ash:1991}.

There are various results using tameness of pseudovarieties to establish the decidability of pseudovarieties obtained by application of the operations of semidirect product, Mal'cev product and join \cite{Almeida&Costa&Teixeira:2010, Almeida&Azevedo&Zeitoun:1997, Almeida&Costa&Zeitoun:2004}.

Also there are connections between tameness and geometry and model theory \cite{Gitik&Rips:1995, Gitik:1999b, Herwig&Lascar:1997, Almeida&Delgado:1997,Almeida&Delgado:1999}. So, it is worth finding more examples of tame pseudovarieties.

It has been established that for every prime numbers $p$, the pseudovariety $\mathsf{G}_p$ of all finite $p$-groups is  tame \cite{Jorge:2002}.
Using this result, in this paper we show that the pseudovariety  $\mathsf{G}_{nil}$ of all finite nilpotent groups is tame with respect to an enlarged implicit signature $\sigma$. Since the free $\sigma$-subalgebra generated by a finite alphabet $A$   is not any more the free  group, in section 3, we  prove the word problem is decidable in this free $\sigma$-subalgebra, meaning that there is an algorithm to decide whether two elements of this $\sigma$-algebra represent the same element.

In the last section, we show that the pseudovariety  $\mathsf{G}_{nil}$ is $\sigma$-reducible if and only of for all prime number $p$, the pseudovariety
$\mathsf{G}_p$ is $\sigma$-reducible. This theorem yields as a corollary  that the pseudovariety $\mathsf{G}_{nil}$ is tame with respect to the systems of equations associated to finite directed graphs but  is not completely tame.

\section{Preliminaries}
%\subsection{Profinite Semigroups}
A \emph{topological semigroup} is a semigroup $S$ endowed with a topology such that the basic semigroup
multiplication $S \times S \rightarrow S$ is continuous. We say that a topological semigroup $S$ is $A$-\emph{generated} if there is a mapping $\varphi : A \rightarrow S$ such that
$\varphi(A)$ generates a dense subsemigroup of $S$.

A \emph{pseudovariety of semigroups} is a class of finite semigroups closed under taking subsemigroups, homomorphic images, and direct products.
Given a pseudovariety  $\mathsf{V}$ of  semigroups, by a pro-$\mathsf{V}$ semigroup $S$ we mean a compact, zero-dimensional semigroup which is residually
in $\mathsf{V}$, that is for
every two distinct points $s, t \in S$, there exists a continuous homomorphism $\varphi : S \rightarrow T$ into some
member $T\in \mathsf{V}$ such that $\varphi(s) \neq \varphi(t)$.

  For  a finite set $A$ in the variety generated by $\mathsf{V}$,  we denote by $\overline{\Omega}_A\mathsf{V}$ the free pro-$\mathsf{V}$ semigroup. The free pro-$\mathsf{V}$ semigroup has the universal property in variety of pro-$\mathsf{V}$ semigroups in the sense   that  for every mapping $\varphi : A \rightarrow S$ into a pro-$\mathsf{V}$
semigroup $S$, there exists a unique continuous homomorphism $\hat{\varphi} : \overline{\Omega}_A\mathsf{V}\rightarrow S$ such that the following
diagram commutes:
$$
\xymatrix{
A
\ar[r]^(.4)\iota
\ar[d]_\varphi
&
\overline{\Omega}_A\mathsf{V}
\ar[ld]^{\hat\varphi}
\\
S
}
$$
%\node (A) {$A$};
For an $A$-generated  pro-$\mathsf{V}$ semigroup $S$, we view $S^A$ both as a direct power of $S$ and as the set of all functions from the set $A$ to $S$. To each element $w\in \overline{\Omega}_A\mathsf{V}$, we may associate an $A$-ary operation  $w_S:S^A\rightarrow S$: for every $\varphi\in S^A$, by the universal property of $\overline{\Omega}_A\mathsf{V}$, there is a unique extension $\hat\varphi:\overline{\Omega}_A\mathsf{V}\rightarrow S$.  Define $w_S(\varphi)=\hat\varphi(w)$. It is easy to see that
for every continuous
homomorphism $f : S \rightarrow T$ between pro-$\mathsf{V}$ semigroups, the
following diagram commutes:
$$
\xymatrix{
S^A
\ar[r]^{w_S}
\ar[d]_{f\circ-}
&
S
\ar[d]^{f}
\\
T^A
\ar[r]^{w_T}
&
T
}
$$
Operations with that property are called $A$-\emph{ary implicit operations}.
The element $w\in \overline{\Omega}_A\mathsf{V}$ is completely determined by the implicit operation $(w_S)_{S\in \mathsf{V}}$ \cite{Jorge:2002}. Note that, the elements of $A$ correspond to the component projects.

We say that an implicit operation $w\in \overline{\Omega}_A\mathsf{V}$ is \emph{computable} if there is an algorithm which given $S\in\mathsf{V}$ and $\varphi\in S^A$, output the value $w_S(\varphi)$.

An $A$-\emph{ary implicit operator} on a pro-$\mathsf{V}$ semigroup $S$ is a transformation $f:S^A\rightarrow S^A$ of $S^A$ to itself whose components $f_i:S^A\rightarrow S$ determined an $A$-ary implicit operation. The set of all $A$-ary implicit operators on a  pro-$\mathsf{V}$ semigroup $S$ is denoted by $\mathcal{O}_A(S)$. The set $\mathcal{O}_A(S)$ is a monoid under composition.% The metric $\overline{\Omega}_A$ induces a metric on $\mathcal{O}_A(S)$:
%$$d(T,U)=\inf\{d(u,v)\mid u,v \in \overline{\Omega}_A, v_S=T_i, u_S=U_i\}$$
\begin{proposition}\cite[Proposition 2.2]{Jorge:2002}
 There is a natural topology on $\mathcal{O}_A(S)$ such that the correspondence
$$S\rightarrow \mathcal{O}_A(S)$$
defines a functor from the category
of pro-$\mathsf{V}$ semigroups with onto continuous homomorphisms as morphisms into the
category of profinite monoids.
\end{proposition}
  For $n$-ary implicit operations  $w_1,\ldots,w_n\in \overline{\Omega}_n\mathsf{V}$ and a pro-$\mathsf{V}$ semigroup $S$, denote by $( w_1,\ldots,w_n)$ the implicit operator
  \begin{equation*}
\begin{aligned}
S^n&\rightarrow S^n\\
(s_1,\ldots,s_n)&\mapsto((w_1)_S(s_1,\ldots,s_n),\ldots,(w_n)_S(s_1,\ldots,s_n)).
\end{aligned}
\end{equation*}
Denote composition of operators by concatenation: $ ( v_1,\ldots,v_n)( w_1,\ldots,w_n)$ has component $i$ determined by the operation $v_i( w_1,\ldots,w_n)$.

Recall that, for an element $v$ of a finite
semigroup $V$, $v^\omega$ denotes the unique idempotent power of $v$. This defines a unary
implicit operation $x \mapsto x^\omega$ on finite semigroups (and similarly on finite monoids)
which therefore has a natural interpretation on each profinite monoid of the form
$\mathcal{O}_n(S)$. %This operation now allows us to set up a rather general recursion scheme
%to construct implicit operations on an arbitrary pseudovariety $\mathsf{V}$ of finite algebras.
Note that $x^\omega$ is the limit of the sequence $\{x^{n!}\}_n$.
We denote by $a_j\circ(w_1,\ldots,w_n)^\omega$, the $j$-th component of the $\omega$-power of the operator $(w_1,\ldots,w_n)$.
\begin{lemma}\cite[Corollary 2.5]{Jorge:2002}
Let $w_1,\ldots, w_n\in \overline{\Omega}_n\mathsf{V}$.
Then each component of $( w_1,\ldots,w_n)^\omega$ is also
a member of $\overline{\Omega}_n\mathsf{V}$. Moreover, if $w_i$ are computable operations, then so is each $a_j\circ( w_1,\ldots,w_n)^\omega$.
\end{lemma}

Let $\mathsf{S}$ be the pseudovariety of all finite semigroups. The elements of $(\overline{\Omega}_A\mathsf{S})^1$,
over arbitrary finite alphabets $A$, are called \emph{pseudowords}.
A \emph{pseudoidentity} is a formal equality of the form $u = v$ with $u, v \in \overline{\Omega}_A\mathsf{S}$ for
some finite alphabet $A$.
For a pseudovariety $\mathsf{V}$ of semigroups, we denote by $\psi_\mathsf{V}$ the unique continuous homomorphism
$\overline{\Omega}_A\mathsf{S}\rightarrow \overline{\Omega}_A\mathsf{V}$ which restricts to the identity on $A$.
We say that $\mathsf{V}$ satisfies the  pseudoidentity $u = v$ with $u, v \in \overline{\Omega}_A\mathsf{S}$ if $\psi_\mathsf{V}(u)=\psi_\mathsf{V}(v)$.

 By an \emph{implicit signature} we mean a set of pseudowords including
multiplication.
An important example is given by the canonical signature $\kappa$ consisting of the multiplication
and the unary operation $x \mapsto x^{\omega-1}$ which, to an element $s$ of a finite semigroup with $n$ elements,
associates the inverse of $s^{1+n!}$ in the cyclic subgroup generated by this power.

 Let $\sigma$ be an implicit signature. Under the natural interpretation of the
elements of $\sigma$, every profinite semigroup may be viewed as a $\sigma$-algebra in the
sense of universal algebra. The $\sigma$-subalgebra of
$\overline{\Omega}_{A}\mathsf{V}$ generated by $A$ is
denoted by ${\Omega}_{A}^{\sigma}\mathsf{V}$ and it is freely generated by $A$. We say that ${\Omega}_{A}^{\sigma}\mathsf{V}$ has  \emph{decidable word problem} if there is an algorithm to decide whether two pseudowords $u, v\in\overline{{\Omega}}_{A}\mathsf{S}$ with $\psi_{\mathsf{V}}(u),\psi_{\mathsf{V}}(v)\in {\Omega}_{A}^{\sigma}\mathsf{V}$   represent the same implicit  operation in   ${\Omega}_{A}^{\sigma}\mathsf{V}$.

For every  subset $L$of $\overline{\Omega}_A\mathsf{S}$ and $\mathsf{V}$ be a pseudovariety, we denote by $Cl(L)$, $Cl_{\mathsf{V}}(L)$, and $Cl_{\sigma,\mathsf{V}}(L)$, the closure of $L$ in   $\overline{\Omega}_{A}\mathsf{S}$, $\overline{\Omega}_{A}\mathsf{V}$,  and ${\Omega}^{\sigma}_{A}\mathsf{V}$, respectively.

As it is mentioned in the introduction, the  property tameness requires the property reducibility. To define  reducibility we need a system of equations.
Let $X$  and $P$ be disjoint  finite sets,
whose elements will be the \emph{variables} and the \emph{parameters}  of the system, respectively.  Consider the following  system of equations:
\begin{equation}\label{eq32}
\begin{aligned}
u_i=v_i&&&&&&&& (i=1\ldots,m),
\end{aligned}
\end{equation}
where $u_i$ and $v_i$ are pseudowords  of $\overline{\Omega}_{X\cup P}\mathsf{S}$. We also fix a finite set $A$ and for every
$x \in X$, we choose a  rational subset $L_x\subseteq A^{*}$. For every parameter
$p \in P$, we associate an element $w_p\in\overline{\Omega}_{A}\mathsf{S}$. A \emph{solution} of the system \eqref{eq32} modulo $\mathsf{V}$ satisfying the constraints is  a function $\delta:{X\cup P}\rightarrow \overline{\Omega}_{A}\mathsf{S}$ satisfying  the following conditions:
\begin{enumerate}
  \item $\delta(x)\in Cl(L_x)$.
 \item $\delta(p)= w_p$.
  \item  $\mathsf{V}$ satisfies  the pseudoidentities  $\delta(u_i)=\delta(v_i)$ ($i=1\ldots,m$).
\end{enumerate}

\begin{theorem}\label{theorem11}\cite[Theorem 5.6.1]{Jorge:1994}
 Let $\mathsf{V}$ be a pseudovariety of finite semigroups.  The following conditions are equivalent for a finite system $\Sigma$ of equations with
rational constraints over the finite alphabet $A$:
\begin{itemize}
  \item $\Sigma$ has a solution modulo every $A$-generated semigroup in $\mathsf{V}$
  \item $\Sigma$ has a solution modulo $\overline{\Omega}_A\mathsf{V}$.
\end{itemize}
\end{theorem}

Let $\sigma$ be an implicit signature.
Consider a system of the form \eqref{eq32}, with   constraints $L_x\subseteq A^*$  and $w_p\in {\Omega}^{\sigma}_{A}\mathsf{S}$ ($x\in X$ and $p\in P$)   where  $u_i$ and $v_i$ are $\sigma$-terms in ${\Omega}^{\sigma}_{X\cup P}\mathsf{S}$. Assume that this system has a solution modulo $\mathsf{V}$. A pseudovariety $\mathsf{V}$  is said to be \emph{$\sigma$-reducible}  for this system if it has a solution $\delta:{X\cup P}\rightarrow {\Omega}_{A}^\sigma \mathsf{S}$ modulo $\mathsf{V}$. We say that $\mathsf{V}$ is \emph{completely
$\sigma$-reducible} if it is $\sigma$-reducible for every such system.

\begin{proposition}\cite[Proposition 3.1]{Jorge&Zeitoun&Carlos:2007}
  Let $\mathsf{V}$ be a pseudovariety. If $\mathsf{V}$ is $\sigma$-reducible
respect to the  systems of equations of $\sigma$-terms without parameters, then $\mathsf{V}$ is completely
$\sigma$-reducible.
\end{proposition}

We say that a pseudovariety $\mathsf{V}$ is \emph{$\sigma$-tame} with respect to a class $\mathfrak{C}$ of systems of equations if
the following conditions hold:
\begin{itemize}
 \item  for every system of equations in $\mathfrak{C}$, the pseudovariety $\mathsf{V}$ is $\sigma$-reducible;
 \item  the word problem is decidable in $\Omega_A^{\sigma}\mathsf{V}$;
\item $\mathsf{V}$ is recursively enumerable, in the sense  that there is some Turing machine which outputs
successively  representatives of all the isomorphism classes of members of $V$.
\end{itemize}
We say that $\mathsf{V}$ is \emph{completely $\sigma$-tame} if $V$ is completely $\sigma$-reducible. Some  important tameness results are as follows:
\begin{itemize}
  \item The pseudovariety $\mathsf{G}$ of all finite groups is $\kappa$-tame with respect to systems of equations associated with finite
directed graphs \cite{Jorge&Ben:2000,Ash:1991}.  It follows from results of Coulbois and Kh\'{e}lif
 that $\mathsf{G}$ is not completely $\kappa$-tame \cite{Coulbois&&Khelif}.
 \item For a prime number $p$, the pseudovariety $\mathsf{G}_p$ of all finite $p$-groups is tame  with respect to the systems of equations associated with finite directed graphs, but not $\kappa$-tame \cite{Jorge:2002}.
  \item The pseudovariety $\mathsf{Ab}$ of all finite abelian groups is completely $\kappa$-tame \cite{Jorge&&Delgado:2005}.
   \end{itemize}

Consider  a system of equations \eqref{eq32} with   constraints $L_x\subseteq A^*$  and $w_p\in {\Omega}^{\sigma}_{A}\mathsf{S}$ ($x\in X$ and $p\in P$)   where  $u_i$ and $v_i$ are $\sigma$-terms in ${\Omega}^{\sigma}_{X\cup P}\mathsf{S}$.  A pseudovariety $\mathsf{V}$ is said to be  \emph{weakly $\sigma$-reducible} with respect to this system if in the case it has a solution modulo $\mathsf{V}$, then there is a solution $\delta:{X\cup P}\rightarrow \overline{\Omega}_{A} \mathsf{S}$ modulo $\mathsf{V}$ which satisfies the conditions $\psi_{\mathsf{V}}(\delta(x))\in {\Omega}_{A}^\sigma \mathsf{V}$ ($x\in X$). It is obvious that if a pseudovariety $\mathsf{V}$ is $\sigma$-reducible, then it is weakly $\sigma$-reducible but the converse is not true. For a prime number $p$, the pseudovariety $\mathsf{G}_p$ of all finite $p$ groups is weakly $\kappa$-reducible but it is not $\kappa$-reducible \cite{Steinberg:1998a, Jorge&Ben:2000}.

We say that a pseudovariety $\mathsf{V}$ is $\sigma$-\emph{full} if for every rational language
$L \subseteq A^*$, the set $\psi_\mathsf{V}(Cl_{\sigma}(L))$ is closed in ${\Omega}_{A}^{\sigma}\mathsf{V}$.

\begin{proposition}\label{proposition3}\cite[Proposition 4.5]{Jorge&Ben:2000}
Every $\sigma$-full weakly  $\sigma$-reducible pseudovariety  is $\sigma$-reducible.
\end{proposition}

Let $x\in \overline{\Omega}_A\mathsf{S}$ and fix $n\in \mathbb{N}$. We denote by $x^{n^\omega}$, the pseudoword  $$\lim_{k\to \infty} x^{n^{k!}}.$$

Let $w_1,\ldots,w_k$ be group words. We denote by $M(w_1,\ldots,w_k)$ the $k\times |A|$ matrix whose $(i,j)$-entry is  the sum of
the exponents in the occurrences of the letter $a_j\in A$ in the word $w_i$.
\begin{theorem}\cite{Jorge:2002}
  Let $p$ be a prime number and  $\sigma_p$ be the set of all  implicit signature  obtained by adding to the
canonical signature $\kappa$ all implicit operations of the form
$$ a_j\circ(w_1,\ldots, w_n)^\omega$$
with $j = 1,\ldots, n$,  $w_i$ are $\kappa$-terms such that the subgroup $\left<w_1,\ldots,w_n\right>$ is $\mathsf{G}_p$-dense. Then  $\mathsf{G}_{p}$ is $\sigma_p$-tame.
\end{theorem}
We use the implicit operations in the preceding  theorem and we prove the following theorem:
\begin{theorem}
Let $\sigma$ be the set of all  implicit signature  obtained by adding to the
canonical signature $\kappa$ all implicit operations of the form
$$ \left(a_j\circ(w_1,\ldots, w_n)^{\omega}\right)^{m^\omega} \ \ \ (m\in \mathbb{N})$$
with $j = 1,\ldots, n$,  $w_i$ are $\kappa$-terms such that $\det{M(w_1,\ldots,w_n)}\ne0$ and every prime number $p$ dividing $\det{M(w_1,\ldots,w_n)}$ also divides $m$. Then  $\mathsf{G}_{nil}$ is $\sigma$-tame.
\end{theorem}
Note that by \cite{Margolis&Sapir&Weil:2001}, for a prime number $p$ and $\kappa$-words $w_1,\ldots, w_n$, the subgroup $\left<w_1,\ldots,w_n\right>$ is $\mathsf{G}_p$-dense if and only if $\det{M(w_1,\ldots,w_n)}\nequiv0 \pmod{p}$ .

\section{The pseudovariety \texorpdfstring{$\mathsf{G}_{nil}$}{TEXT} is \texorpdfstring{$\sigma$}{TEX}-tame}
\subsection{The word problem is decidable  in \texorpdfstring{$\Omega_A^{\sigma}\mathsf{G}_{nil}$}{TEXT}}
We use the following lemmas to reduced the word problem in $\Omega_A^{\sigma}\mathsf{G}_{nil}$ to the the word problem in the free group generated by $A$.
\begin{lemma}\label{lemma53}
Let $u,v\in \overline{\Omega}_A\mathsf{S}$. Then the pseudoidentity  $u=v$ is valid in $\mathsf{G}_{nil}$ if and only if,  for every prime number $p$, the pseudoidentity  $u=v$ is valid in $\mathsf{G}_p$.
\end{lemma}
\begin{proof}
Since every $p$-group is a nilpotent group, if $\mathsf{G}_{nil}$ satisfies the pseudoidentity  $u=v$, then for every prime number $p$, $\mathsf{G}_p$ satisfies the pseudoidentity  $u=v$.

 The converse follows from fact that every finite nilpotent
group is isomorphic to the direct product of its $p$-Sylow subgroups.
\end{proof}

We denote by $\mathds{P}$, the set of all prime numbers.
\begin{lemma}\label{lemma54}
Fix $p\in \mathds{P}$. The pseudoidentity  $x^{n^\omega}=x$ is valid in $\mathsf{G}_p$ if $n$ is not divisible by $p$ and the
pseudoidentity $x^{n^\omega}=1$ is valid in $\mathsf{G}_p$ otherwise.
\end{lemma}
\begin{proof}
The result follows from the elementary Euler congruence theorem.
\end{proof}

\begin{theorem}\label{th1}
For every pseudoword $u\in \Omega_A^{\sigma}\mathsf{S}$, there is a computable  cofinite subset $S(u)$ of $\mathds{P}$ such that  the following properties hold:
\begin{itemize}
  \item there is a computable $\kappa$-word $w_0$ such that for every $p\in S(u)$, the pseudoidentity $u=w_0$ is valid in $\mathsf{G}_p$;
  \item for every $p_i\in\{p_1,\ldots,p_r\}=\mathds{P}\setminus S(u)$, there is a computable  $\kappa$-word $w_i$ such that  the pseudoidentity $u=w_i$ is valid in $\mathsf{G}_{p_i}$ $(1\leq i\leq r)$.
\end{itemize}
In particular, for every prime number $p$, we have $\Omega^{\sigma}_A\mathsf{G}_{p}=\Omega^{\kappa}_A\mathsf{G}_{p}$.
\end{theorem}
\begin{proof}
Since every $u\in \Omega_A^{\sigma}\mathsf{S}$ is constructed from the letters in $A$ using a finite number of times  the operations  in $\sigma$ and the intersection of a finite number of cofinite sets is cofinite,
it is enough to show that the statement of the theorem holds for  every pseudoword $u\in \sigma\setminus \kappa$.

  Consider the following pseudoword in $\sigma\setminus \kappa$
$$u= \left(a_j\circ(w_1,\ldots, w_n)^{\omega}\right)^{m^\omega} \ \ \ (m\in \mathbb{N})$$
where $j = 1,\ldots, n$,  $w_i$ are $\kappa$-terms such that $\det{M(w_1,\ldots,w_n)}\ne0$ and every prime number $p$ dividing $\det{M(w_1,\ldots,w_n)}$ also divides $m$. Let
$$S(u)=\mathds{P}\setminus\{p\mid p \ \text{divides} \ m\}.$$
If $p$ lies in $S(u)$, then by \cite[Corollary 3.3]{Margolis&Sapir&Weil:2001}, the subgroup $\left<w_1,\ldots,w_n\right>$ is $\mathsf{G}_p$-dense in $\overline{\Omega}_A\mathsf{G}_p$. Hence, by \cite[Lemma 6.5]{Jorge:2002}, the following pseudoidentity holds in $\mathsf{G}_p$
$$a_j\circ(w_1,\ldots, w_n)^{\omega}=a_j.$$
 Consider $w_0=a_j$.

  Otherwise, $p$ divides $m$ and therefore by Lemma \ref{lemma54}, the pseudoidentity $$\left(a_j\circ(w_1,\ldots, w_n)^{\omega}\right)^{m^\omega}=1$$ holds in $\mathsf{G}_p$.
 \end{proof}
\begin{corollary}
The word problem is decidable in $\Omega^{\sigma}_A\mathsf{G}_{nil}$.
\end{corollary}
\begin{proof}
Let $u,v\in \Omega^{\sigma}_A\mathsf{S}$. Then by Lemma \ref{lemma53}, the pseudoidentity $u=v$ holds in $\mathsf{G}_{nil}$ if and only if for every prime number $p$, the pseudoidentity $u=v$ holds in $\mathsf{G}_{p}$.

Since $\mathsf{G}_{p}$ is a pseudovariety of finite groups, the pseudoidentity $u=v$ is valid in $\mathsf{G}_{p}$ if and only if the pseudoidentity $uv^{\omega-1}=1$ is valid in $\mathsf{G}_{p}$. Now  the result follows from the preceding lemma and the fact that the word problem is decidable in the free group.
\end{proof}

\subsection{The pseudovariety  \texorpdfstring{$\mathsf{G}_{nil}$}{TEXT}
  is \texorpdfstring{$\sigma$} {TEXT}-reducible}

We denote by $FG(A)$, the free group over $A$  and
for a finitely generated (f.g.)~subgroup $H$ of a free group, denote by $\mathds{P}(H)$, the set of all prime numbers $p$ such that $H$ is $\mathsf{G}_p$-closed.

\begin{proposition}\cite[proposition 4.3]{Margolis&Sapir&Weil:2001}
Let $H$ be a f.g.~subgroup of $FG(A)$. The set $\mathds{P}(H)$
is a  finite or a cofinite subset of $\mathds{P}$, and it is effectively computable.
\end{proposition}
\begin{lemma}\label{lemma51}
Let $H$ be a f.g.~subgroup of $FG(A)$. Then there is a cofinite subset $S(H)$ of $\mathds{P}$ and a f.g.~subgroup $K$ of $FG(A)$ such that for every $p\in S(H)$, $Cl_{\mathsf{G}_p,\kappa}(H)=K$.
\end{lemma}
\begin{proof}
 Let $K_1,\ldots,K_m$ be the set of all overgroups of the subgroup $H$ (i.e., the automaton of $K_i$ is a quotient of the automaton of $H$). For every prime number $p$, there is $i$ ($1\leq i\leq m$) such that $Cl_{\mathsf{G}_p,\kappa}(H)=K_i$ \cite{Margolis&Sapir&Weil:2001}.  Since the number of overgroups of $H$ is a finite set, there is $K_{j}$  and  an infinite subset $S$ of $\mathds{P}$   such that for every $p\in S$, we have $Cl_{\mathsf{G}_p,\kappa}(H)=K_{j}$ $(1\leq j\leq m)$. If $S$ is cofinite, then we are done. For every $p\in S$, we have
$$K_{j}\subseteq Cl_{\mathsf{G}_p,\kappa}(K_{j})\subseteq Cl_{\mathsf{G}_p,\kappa}(Cl_{\mathsf{G}_p,\kappa}(H))= Cl_{\mathsf{G}_p,\kappa}(H)=K_{j}$$
 and hence, $S\subseteq \mathds{P}(K_{j})$. Therefore, by the preceding proposition, $\mathds{P}(K_{j})$ is a cofinite subset of $\mathds{P}$.

 Suppose that there is an overgroup $L$ of $H$ properly contained in $K_j$ such that $\mathds{P}(L)$ is a cofinite subset of $\mathds{P}$. Then for every $p\in \mathds{P}(L)$, we have
 $$Cl_{\mathsf{G}_p,\kappa}(H)\subseteq  Cl_{\mathsf{G}_p,\kappa}(L)=L\varsubsetneq K_j.$$
 Hence, the set of all prime numbers $q$ such that $Cl_{\mathsf{G}_q,\kappa}(H)=K_j$ is contained in the set $\mathds{P}\setminus \mathds{P}(L)$. Since we assume that $\mathds{P}(L)$ is a cofinite subset of $\mathds{P}$, the set $\mathds{P}\setminus \mathds{P}(L)$ is a finite set which contradicts the choice of $K_j$.

 So, for every overgroup $L$ of $H$ properly  contained in $K_{j}$, $\mathds{P}(L)$ is a finite set. Let  $K=K_{j}$ and
 \begin{equation*}
   S(H)=\mathds{P}(K)\setminus\bigcup_{\mathclap{\substack{L\  \text{overgroup of}\ H \\
   L\subset K}}} \mathds{P}(L).
   \end{equation*}
   As the number of over groups of $H$ is finite,  $S(H)$ is a cofinite subset of $\mathds{P}$. For every $p\in S(H)$,  we have
 $$Cl_{\mathsf{G}_p,\kappa}(H)\subseteq Cl_{\mathsf{G}_p,\kappa}(K)=K,$$
 hence, $Cl_{\mathsf{G}_p,\kappa}(H)$ is a $\mathsf{G}_p$-closed overgroup of $H$ contained in $K$,   by the choice of $p$, it follows that $Cl_{\mathsf{G}_p,\kappa}(H)=K$.
\end{proof}

The following two propositions are the main tools to show that the pseudovariety  $\mathsf{G}_{nil}$ is  $\sigma$-reducible.
\begin{proposition}\label{lemma52}
Let $H_1,\ldots, H_t$ be  f.g.~subgroups of the free group and fix  the cofinite sets $S(H_i)$ and the subgroups $K_i$ as in the preceding lemma. Let $w\in K_1\ldots K_t$. Then there are $u_i\in \Omega_A^{\sigma}\mathsf{S}\cap Cl(H_i)$ and  cofinite subsets $S_1(H_i)$ of $\mathds{P}$ contained in $S(H_i)$ such that, if $p\in \displaystyle\bigcap_{i=1}^tS_1(H_i)$, then the pseudoidentity $w=u_1\ldots u_t$ is valid in $\mathsf{G}_p$ and the pseudoidentity $u_i=1$ is valid in $\mathsf{G}_p$ otherwise.
 \end{proposition}
 \begin{proof}
Since, for every prime number $p$, $\mathsf{G}_p$ is an extension-closed pseudovariety,  by \cite[Proposition 2.10]{Margolis&Sapir&Weil:2001}, there are $\kappa$-words $w_{1,i},\ldots,w_{s_i,i}$ such that  $K_i=\left<w_{1,i}\ldots,w_{s_i,i}\right>$ and $rk(K_i)\leq rk(H_i)$ ($1\leq i \leq t$).

Fix $q\in \displaystyle\bigcap_{i=1}^tS(H_i)$.
  As $K_i$ contains $H_i$, we can rewrite the generator $h$ of $H_i$ as a reduced word $h'$  in terms of generators of $K_i$.  Since $Cl_{\mathsf{G}_q,\kappa}(H_i)=K_i$, by \cite[Proposition 2.9]{Margolis&Sapir&Weil:2001}  $H_i$ is $\mathsf{G}_q$-dense in the pro-$\mathsf{G}_q$ topology on $K_i$. Hence, by  \cite[Proposition  5.2]{Jorge:2002}, there is a subset  $\{h_{1,i}',\ldots,h_{s_{i},i}'\}$ of generators of $H_i$ such that the subgroup $H_i'$ generated by $\{h_{1,i}',\ldots,h_{s_{i},i}'\}$  is $\mathsf{G}_q$-dense  in the pro-$\mathsf{G}_q$ topology on $K_i$.  Let $M(h_{1,i}',\ldots,h_{s_i,i}')$ be the $s_i\times s_i$ matrix  whose $(k,j)$-entry is the number of occurrences of $w_{j,i}$ in $h_{k,i}'$. Then by \cite[Corollary 3.3]{Margolis&Sapir&Weil:2001}, we have $\det{ M(h_{1,i}',\ldots,h_{s_i,i}')}\nequiv 0 \pmod{q}$ and also for every prime number $p$ not dividing $\det{M(h_{1,i}',\ldots,h_{s_i,i}')}$, the subgroup $H_i'$ is $p$-dense in the pro-$\mathsf{G}_p$ topology on $K_i$. Since for every $p\in S(H_i)$, $K_i$ is $\mathsf{G}_p$-closed, by \cite[Proposition 2.9]{Margolis&Sapir&Weil:2001} for every $p\in S(H_i)$ not dividing  $\det{M(h_{1,i}',\ldots,h_{s_i,i}')}$ we have $Cl_{\mathsf{G}_p,\kappa}(H_i')=Cl_{\mathsf{G}_p,K_i}(H_i')=K_i$, where $Cl_{\mathsf{G}_p,K_i}(H_i')$ is the $p$-closure of $H_i'$ in the pro-$\mathsf{G}_p$ topology on $K_i$. Hence, by the proof of \cite[Theorem 6.1]{Jorge:2002}, for every $p\in S(H_i)$ not dividing $\det{M(h_{1,i}',\ldots,h_{s_i,i}')}$, the following pseudoidentity holds in $\mathsf{G}_p$
\begin{equation}\label{01}
  w_{j,i}=a_j\circ(h_{1,i},\ldots,h_{s_i,i})^\omega
\end{equation}
  where $h_{k,i}$  corresponds  to  $h_{k,i}'$ as a word in terms of the alphabet $A$.
  Note that $a_j\circ(h_{1,i},\ldots,h_{s_i,i})^\omega$ belongs to $Cl(H_i)\subseteq \overline{\Omega}_A\mathsf{S}$.

   Let
 $$S_1(H_i)=S(H_i)\setminus\{p\mid  \det{M(h_{1,i}',\ldots,h_{s_i,i}')} \equiv0 \pmod{p}\}.$$
 As  $\det{M(h_{1,i}',\ldots,h_{s_i,i}')}\neq 0$, the  sets  $S_1(H_i)$ are  cofinite subsets of $\mathds{P}$.

 Let $P_0=\{p_1,\ldots,p_r\}=\mathds{P}\setminus \bigcap_{i=1}^rS_1(H)$ and
 $n_0=p_1\ldots p_r$. Consider the following pseudowords:
 $$u_{j,i}= \left(a_j\circ(h_{1,i},\ldots,h_{s_i,i})^\omega\right)^{{n_0}^\omega}\in \Omega_A^{\sigma}\mathsf{S}\cap Cl(H_i).$$
 If $p\in \mathds{P}\setminus P_0$,  by Lemma \ref{lemma54}, the following pseudoidentities  are valid in $\mathsf{G}_{p}$:
\begin{equation*}
  \begin{aligned}
  u_{j,i}&= \left(a_j\circ(h_{1,i},\ldots,h_{s_i,i})^\omega\right)^{{n_0}^\omega}=a_j\circ(h_{1,i},\ldots,h_{s_i,i})^\omega\stackrel{\eqref{01}}{=}w_{j,i}.
  \end{aligned}
\end{equation*}
  Otherwise, the following pseudoidentities are valid in $\mathsf{G}_{p}$ ($p\in P_0$):
\begin{equation*}
  \begin{aligned}
  u_{j,i}&=\left(a_j\circ(h_{1,i},\ldots,h_{s_i,i})^\omega\right)^{{n_0}^\omega}=1.
  \end{aligned}
\end{equation*}
So,  we showed that for every generator $w_{k,i}$ of $K_{i}$, there are $u_{k,i}\in  \Omega_A^{\sigma}\mathsf{S} \cap  Cl(H_i)$  such  that, if  $p\in \bigcap_{i=1}^rS_1(H_i)$, then the pseudoidentity $u_{k,i}=w_{k,i}$ is valid in $\mathsf{G}_{{p}}$  and the pseudoidentity $u_{k,i}=1$ is valid in $\mathsf{G}_{{p}}$ otherwise. Hence, every $w\in K_i$ has this property ($1\leq i\leq t$).

Let $w\in K_1\ldots K_t$. Then there are $w_i\in K_i$ such that $w=w_1\ldots w_t$. By the preceding paragraph, there are $u_{i}\in  \Omega_A^{\sigma}\mathsf{S}\cap Cl(H_i)$  such  if  $p\in \bigcap_{i=1}^rS_1(H_i)$, then the pseudoidentity $u_{i}=w_i$ is valid in $\mathsf{G}_{{p}}$  and the pseudoidentity $u_{i}=1$ is valid in $\mathsf{G}_{{p}}$ otherwise. Let
$v=u_1\ldots u_t\in \Omega_A^{\sigma}\mathsf{S}\cap Cl(H_1\ldots H_t).$

If $p\in  \bigcap_{i=1}^rS_1(H_i)$, then the pseudoidentity  $v=w$ is valid in $\mathsf{G}_{{p}}$   and the pseudoidentity $v=1$ is valid in $\mathsf{G}_{{p}}$ otherwise.
 \end{proof}

 \begin{proposition}\label{proposition2}
 Let $H_1,\ldots, H_t$ be  f.g.~subgroups of the free group and fix  $S_1(H_i)$ as in the preceding proposition.
Then for every $p$ in the set
$$\mathds{P}\setminus \displaystyle\bigcap_{i=1}^t S_1(H_i)=\{p_1\ldots,p_r\}$$ and
every $w\in Cl_{\mathsf{G}_{p},\kappa}(H_1\ldots H_t)$, there are $u_i\in\Omega_A^{\sigma}\mathsf{S}\cap Cl(H_i)$  such that  the pseudoidentity $u_{i}=1$ is valid in $\mathsf{G}_q$ for every $q\in \mathds{P}\setminus \{p\}$ and the pseudoidentity $u_1\ldots u_t=w$ is valid in $\mathsf{G}_{p}$.
\end{proposition}

 \begin{proof}
 By \cite[Lemma 5.2]{Ribes&Zalesskii:1993b}, for every prime number $p$, we have
 \begin{equation}\label{eqc}
   Cl_{\mathsf{G}_{p},\kappa}(H_1\ldots H_t)=Cl_{\mathsf{G}_{p},\kappa}(H_1)\ldots Cl_{\mathsf{G}_{p},\kappa}(H_t).
 \end{equation}

 Let $P_0=\{p_1,\ldots,p_r\}$ and  $K_{p_i,j}=Cl_{\mathsf{G}_{p_i},\kappa}(H_j)$ ($1\leq i\leq r$ and $1\leq j\leq t$).
There are $\kappa$-words $w_{1,p_i,j},\ldots, w_{s_{p_i,j},p_i,j}$ such that $$K_{p_i,j}=\left<w_{1,p_i,j},\ldots, w_{s_{p_i,j},i,j}\right>.$$ By  the proof of  \cite[Theorem 6.1]{Jorge:2002}, there is a subset $\{h_{1,p_i,j},\ldots,h_{s_{p_i,j},p_i,j}\}$ of the generators of $H_j$ such that $$Cl_{\mathsf{G}_{p_i},\kappa}(\left<h_{1,p_i,j},\ldots,h_{s_{p_i,j},p_i,j}\right>)= Cl_{\mathsf{G}_{p_i},\kappa}(H_j)=K_{p_i,j},$$ and
the pseudoidentity
 \begin{equation}\label{eq23}
w_{k,p_i,j}=a_k\circ(h_{1,p_i,j},\ldots,h_{s_{p_i,j},p_i,j})^\omega\ \ \ \ (1\leq k\leq s_{p_i,j})
\end{equation}
is valid in $\mathsf{G}_{p_i}$. We consider the following finite subsets of $\mathds{P}$:
\begin{equation*}
  \begin{aligned}
  P_i=&\bigcup_{j=1}^{t}\{p\mid p \ \text{divides}\ \det{M(h_{1,p_i,j}',\ldots,h_{s_{p_i,j},p_i,j}')}\} &&&&(1\leq i\leq r)
  \end{aligned}
\end{equation*}
where $h_{k,p_i,j}'$  is generator $ h_{k,p_i,j}$  written  in terms of generators of $K_{p_i,j}$ ($1\leq k\leq s_{p_i,j}$). Note that, since $K_{p_i,j}$ is a $p_i$-closed subgroup and
$$Cl_{\mathsf{G}_{p_i},\kappa}(\left<h_{1,p_i,j},\ldots,h_{s_{p_i,j},p_i,j}\right>)=K_{p_i,j},$$ by \cite[Proposition 2.9]{Margolis&Sapir&Weil:2001} the subgroup $\left<h_{1,p_i,j}',\ldots,h_{s_{p_i,j},p_i,j}'\right>$ is $p_i$-dense in the pro-$\mathsf{G}_{p_i}$ topology on $K_{p_i,j}$ and, therefore, $p_i$ does not belong to $P_i$.
Consider the following natural numbers:
\begin{equation*}
  \begin{aligned}
  n_0=&p_1\ldots p_r,\\
  n_i=&\prod_{j=0}^r\prod_{p\in P_j\setminus \{p_i\}} p &&&&&&&&(1\leq i\leq r).
  \end{aligned}
\end{equation*}
  The natural numbers $n_i$ ($1\leq i\leq r$) satisfy the following properties:
  \begin{enumerate}
    \item The prime number $p_i$ does not divide $n_i$.
    \item Fix $p_i\in P_0$. For every $p_j\in P_0$ ($j\neq i$), $p_j$ divides $n_i$, because $p_j$ belongs to the set $P_0\setminus \{p_i\}$.
  \item Since we have $p_i\notin P_i$, every prime number $p\in P_i$ divides $n_i$.
  \item For every prime number $p$ in $\left(P_1\cup\ldots\cup P_r\right)\setminus P_0$, $p$ divides $n_i$ ($1\leq i\leq r$), because there is $j$ such that $p\in P_j$ and since $p$ does not belong to $P_0$, $p$ is in $P_j\setminus \{p_i\}$.
  \end{enumerate}
   For every $i$, $j$, and $k$ ($1\leq j\leq s_{p_i,j}$, $1\leq i\leq r$, and $1\leq j\leq t$), we let
  \begin{equation*}
  \begin{aligned}
    u_{k,p_i,j}&= \left(a_k\circ(h_{1,p_i,j},\ldots,h_{s_{p_i,j},p_i,j})^\omega\right)^{{n_i}^\omega}\\
    &\left(a_k\circ(h_{1,p_i,j},\ldots,h_{s_{p_i,j},p_i,j})^\omega\right)^{(\omega-1){(p_in_i)}^\omega}.
    \end{aligned}
  \end{equation*}
  Note that $u_{k,p_i,j}$ belongs to $ \Omega_A^{\sigma}\mathsf{S}\cap Cl(H_j)$. We claim that  the pseudoidentity  $u_{k,p_i,j}=w_{k,p_i,j}$ is valid  in $\mathsf{G}_{p_i}$ and for every $p\in \mathds{P}\setminus\{p_i\}$ and the pseudoidentity $u_{k,p_i,j}=1$ is valid in $\mathsf{G}_p$ ($1\leq k\leq s_{p_i,j}$ and $1\leq j\leq t$).

It remains to establish the  claim. Consider the following cases:
  \begin{itemize}
    \item Let $p=p_i$. By the property (1) of $n_i$, $p_i$ does  not divides $n_i$. Hence, we have the following pseudoidentities in $\mathsf{G}_{p_i}$
    \begin{equation*}
      \begin{aligned}
        u_{k,p_i,j}=& \left(a_k\circ(h_{1,p_i,j},\ldots,h_{s_{p_i,j},p_i,j})^\omega\right)^{{n_i}^\omega} \\ &\left(a_k\circ(h_{1,p_i,j},\ldots,h_{s_{p_i,j},p_i,j})^\omega\right)^{(\omega-1){(p_in_i)}^\omega}\\
        =& \left(a_k\circ(h_{1,p_i,j},\ldots,h_{s_{p_i,j},p_i,j})^\omega\right)1 \stackrel{\eqref{eq23}}{=}w_{k,p_i,j}.
      \end{aligned}
    \end{equation*}
    \item Consider either  $p\in P_0\setminus \{p_i\}$ or  $p\in (P_1\cup\ldots\cup P_r)\setminus P_0$ . By the property (2) and (4) of $n_i$, $p$ divides $n_i$. Hence, we have the following pseudoidentities in $\mathsf{G}_{p}$
    \begin{equation*}
      \begin{aligned}
         \left(a_k\circ(h_{1,p_i,j},\ldots,h_{s_{p_i,j},p_i,j})^\omega\right)^{{n_i}^\omega} =&1,\\ \left(a_k\circ(h_{1,p_i,j},\ldots,h_{s_{p_i,j},p_i,j})^\omega\right)^{(\omega-1){(p_in_i)}^\omega}=&1.
        \end{aligned}
    \end{equation*}
     So, the pseudoidentity $u_{k,p_i,j}=1$ is valid in $\mathsf{G}_p$.
 %   \item Let . By the property (4) of $n_i$, $p$ divides $n_i$. Hence,  the pseudoidentity $u_{k,p_i,j}=1$ holds in $\mathbf{G}_p$.
    \item Consider $p\in \mathds{P}\setminus (P_0\cup\ldots\cup P_r) $. By the choice of $n_i$, $p$ does not divide $p_in_i$ and, therefore, does not divide  $n_i$. Hence, we have the following pseudoidentities  in $\mathsf{G}_{p}$
    \begin{equation*}
      \begin{aligned}
       u_{k,p_i,j}=&  \left(a_k\circ(h_{1,p_i,j},\ldots,h_{s_{p_i,j},p_i,j})^\omega\right)^{{n_i}^\omega}\\   &\left(a_k\circ(h_{1,p_i,j},\ldots,h_{s_{p_i,j},p_i,j})^\omega\right)^{(\omega-1){(p_in_i)}^\omega}\\
       =& \left( a_k\circ(h_{1,p_i,j},\ldots,h_{s_{p_i,j},p_i,j})^\omega\right)\\
       &\left(a_k\circ(h_{1,p_i,j},\ldots,h_{s_{p_i,j},p_i,j})^\omega\right)^{\omega-1}\\=&1.
        \end{aligned}
    \end{equation*}
 This completes the proof of the claim.
  \end{itemize}
 We showed  for every generator $w_{k,p_i,j}$ of $Cl_{\mathsf{G}_{p_i},\kappa}(H_j)=K_{p_i,j}$, there are $u_{k,p_i,j}\in  \Omega_A^{\sigma}\mathsf{S}\cap Cl(H_j)$  such  that the pseudoidentity  $u_{k,p_i,j}=w_{k,p_i,j}$ is valid in $\mathsf{G}_{{p}_i}$ and the pseudoidentity  $u_{k,p_i,j}=1$ is valid in $\mathsf{G}_{{p}}$ ($ p\in \mathds{P}\setminus \{p_i\}$). Hence, every $w\in Cl_{\mathsf{G}_{p_i},\kappa}(H_j)$ has this property.

 Fix $p\in P_0$ and let $w\in Cl_{\mathsf{G}_p,\kappa}(H_1\ldots H_t)$. By \eqref{eqc}, there are $w_{j}\in Cl_{\mathsf{G}_p,\kappa}(H_j)$ such that $w=w_1\ldots w_t$. By the preceding paragraph, there are $u_{j}\in \Omega_A^{\sigma}\mathsf{S}\cap Cl(H_j)$  such  that the pseudoidentity  $u_{j}=w_j$ is valid in $\mathsf{G}_{{p}}$ and the pseudoidentity $u_{j}=1$ is valid in $\mathsf{G}_{{q}}$ ($q\in \mathds{P}\setminus \{p\}$). Let
  $$v=u_{1}\ldots u_{t}\in \Omega_A^{\sigma}\mathsf{S}\cap Cl(H_1\ldots H_t).$$
   Then the equality $v=w$ is valid in $\mathsf{G}_{p}$ and the equality $v=1$  is valid in $\mathsf{G}_{{q}}$ ($ q\in \mathds{P}\setminus \{p\}$).
\end{proof}

\begin{corollary}\label{cor1}
For every prime number $p$, the pseudovariety $\mathsf{G}_p$ is $\sigma$-reducible with respect to the systems of equations associated with finite directed  graphs.
\end{corollary}
\begin{proof}
By  Theorem \ref{th1}, for a rational subset $L$ of $A^*$ we have
$$Cl_{\mathsf{G}_p,\sigma}(\psi_{\mathsf{G}_p}(L))=Cl_{\mathsf{G}_p,\kappa}(\psi_{\mathsf{G}_p}(L)).$$
 Since the pseudovariety $\mathsf{G}_p$ is weakly $\kappa$-reducible with respect to the systems of equations associated with finite directed  graphs \cite{Steinberg:1998a}, by Theorem \ref{proposition3} we just need to show that   $\mathsf{G}_p$ is $\sigma$-full.

By definition of $\sigma$-full pseudovariety, it is enough to  show that for a rational subset of $A^*$:
\begin{equation}\label{eq31}
  \psi_{\mathsf{G}_p}(Cl_\sigma(L))\subseteq Cl_{\mathsf{G}_p,\kappa}(\psi_{\mathsf{G}_p}(L)),
\end{equation}
or equivalently, it is enough to show that
\begin{equation}\label{eq40}
 w\in  Cl_{\mathsf{G}_p,\kappa}(\psi_{\mathsf{G}_p}(L))\Rightarrow Cl(L)\cap \Omega_A^\sigma \mathsf{S}\cap (\psi^{-1}(w))\neq \emptyset
\end{equation}
 %The reverse inclusion is trivial. Hence, $\mathsf{G}_p$ is $\sigma$-full.
 The proof is similar to the proof of \cite[Theorem 6.1]{Jorge:2002}. Since any rational subset of the free semigroup can be obtained by taking a finite number of finite subsets of the free semigroup and applying the union, product and the plus operation $L\rightarrow L^{+}$  a finite number of times, it is enough to show that these operations preserve this  property.
  As it is  mentioned in the proof of \cite[Theorem 6.2]{Jorge:2002}, for the rational subsets $L$ and $K$, we have
\begin{enumerate}
  \item   $Cl_{\mathsf{G}_p,\kappa}(\psi_{\mathsf{G}_p}(L))=\psi_{\mathsf{G}_p}(L)$ if $L$ is finite;
 \item $Cl_{\mathsf{G}_p,\kappa}(\psi_{\mathsf{G}_p}(LK))=Cl_{\mathsf{G}_p,\kappa}(\psi_{\mathsf{G}_p}(L))Cl_{\mathsf{G}_p,\kappa}(\psi_{\mathsf{G}_p}(K))$;
  \item $Cl_{\mathsf{G}_p,\kappa}(\psi_{\mathsf{G}_p}(L\cup K))=Cl_{\mathsf{G}_p,\kappa}(\psi_{\mathsf{G}_p}(L))\cup Cl_{\mathsf{G}_p,\kappa}(\psi_{\mathsf{G}_p}(L))$
  \item $Cl_{\mathsf{G}_p,\kappa}(\psi_{\mathsf{G}_p}(L)^+)= Cl_{\mathsf{G}_p,\kappa}(\left<\psi_{\mathsf{G}_p}(L)\right>)$
\end{enumerate}
If a language $L$ is finite, then any $w$
in $Cl_{\mathsf{G}_p,\kappa}(\psi_{\mathsf{G}_p}(L))$ is also an element of $Cl(L)\cap \Omega_A^\sigma \mathsf{S}\cap (\psi^{-1}(w))$ and so finite languages satisfy \eqref{eq40}.
Suppose that $L_1$ and $L_2$ are rational subsets of $\Omega_A\mathsf{S}$ satisfy Property \eqref{eq31}. By property (3), at least one of the sets $Cl(L_1)\cap \Omega_A^\sigma \mathsf{S}\cap (\psi^{-1}(w))$ and $Cl(L_2)\cap \Omega_A^\sigma \mathsf{S}\cap (\psi^{-1}(w))$ is nonempty and,
therefore, so is their union.

Let $u\in \psi_{\mathsf{G}_p}(Cl(L_1L_2))\cap \Omega_A^{\sigma}\mathsf{G}_p$. By property (2), there are
$$u_i\in Cl_{\kappa,\mathsf{G}_p}( \psi_{\mathsf{G}_p}(L_i))\ \ (i=1,2)$$
 such that $u=u_1u_2$. By the
 induction hypotheses, there are $w_i\in \psi_{\mathsf{G}_p}^{-1}(u_i)\cap\Omega_A^{\sigma}\mathsf{S}\cap Cl(L_i)$. Let $w=w_1w_2\in \Omega_A^{\sigma}\mathsf{S}\cap Cl(L_1)Cl(L_2)\subseteq\Omega_A^{\sigma}\mathsf{S}\cap Cl(L_1L_2)$.  Then we have
 $$\psi_{\mathsf{G}_p}(w)=\psi_{\mathsf{G}_p}(w_1w_2)=\psi_{\mathsf{G}_p}(w_1)\psi_{\mathsf{G}_p}(w_2)=u_1u_2=u.$$
 Thus,  $w$ belongs to $\Omega_A^{\sigma}\mathsf{S}\cap Cl(L_1L_2)\cap \psi_{\mathsf{G}_p}^{-1}(u)$.

Let $L$ be a rational subset satisfies the property \eqref{eq40}. By \cite[Lemma 6.6]{Jorge:2002}, the is a finite subset $Y$ of $\kappa$ words, such that $\left<L\right>=\left<Y\right>$ and $Cl(Y)\subseteq CL(L)$. Hence, to show that a rational language of the form $L^+$ satisfies \eqref{eq40}, it is enough to establish
that, for any finite subset $Y$ of $\kappa$ words the following property holds:
$$ w\in  Cl_{\mathsf{G}_p,\kappa}(\left<Y\right>)\Rightarrow Cl(Y^+)\cap \Omega_A^\sigma \mathsf{S}\cap (\psi^{-1}(w))\neq \emptyset$$
 The result follows from the Propositions  \ref{lemma52} and \ref{proposition2} by $t=1$.
 \end{proof}

\begin{theorem}
  Let $\mathcal{C}$ be a class of systems of equations. For every prime number $p$, the pseudovariety $\mathsf{G}_p$ is  $\sigma$-reducible with respect to the   systems of equations in $\mathcal{C}$ if and only if  the pseudovariety $\mathsf{G}_{nil}$ is $\sigma$-reducible with respect to the   systems of equations in $\mathcal{C}$
\end{theorem}
\begin{proof}
Let $\mathsf{V}$ and $ \mathsf{W}$ be pseudovarieties of semigroups such that $\mathsf{V}\subseteq \mathsf{W}$. If a pseudoidentity is valid in $\mathsf{W}$, then it is valid in $\mathsf{V}$. Hence, if a equation has solution $\delta$ modulo $\mathsf{W}$, then $\delta$ is a solution modulo $\mathsf{V}$.
%If  the pseudovariety $\mathsf{G}_{nil}$ is $\sigma$-reducible with respect to the   systems of equations in $\mathcal{C}$, then
%for every prime number $p$, the pseudovariety $\mathsf{G}_p$ is  $\sigma$-reducible with respect to the class $\mathcal{C}$.

Conversely,   let $X$ be a finite set of variable and let $k$ be the cardinality of $X$. Consider  a system of equations in $\mathcal{C}$ of the form
\begin{equation}\label{1}
  u_i=v_i \ \ \ (i=1,\ldots,m),
\end{equation}
with rational constraints  $L_x\subseteq (\Omega_{A}\mathsf{S})^1$ and  $u_i,v_i\in\Omega_{X}^\sigma \mathsf{S}$. Suppose that this system  has the solution $\delta:X\rightarrow \overline{\Omega}_A\mathsf{S}$ modulo $\mathsf{G}_{nil}$. Since $L_x$ is a rational subset of $(\Omega_{A}\mathsf{S})^1$, $L_x$ is a finite union of sets of the form $R_0^{*}w_0R_1^{*}\ldots w_tR_t^{*}$, where $R_i$ are rational subsets of $(\Omega_{A}\mathsf{S})^1$ and $w_i\in (\Omega_{A}\mathsf{S})^1$. Hence, for every $x\in X$, there is a simple rational subset
$$L_x'=R_{0,x}^{*}u_{0,x}R_{1,x}^{*}\ldots u_{t_x-1,x}R_{t_x,x}^{*}\subseteq L_x$$
  such that $\delta(x)\in Cl_{\mathsf{G}_{nil}}(L_x')$. So, the system \eqref{1} with constraints  $L'_x\subseteq \Omega_{A}\mathsf{S}$ has the solution $\delta$ modulo $\mathsf{G}_{nil}$ and, therefore, for every prime number $p$, the system \eqref{1} with constraints  $L'_x\subseteq \Omega_{A}\mathsf{S}$ has a solution modulo $\mathsf{G}_{p}$.
  We show that the system \eqref{1} with constraints $L_x'$ ($x\in X$) has a solution $\delta':X\rightarrow {\Omega}^\sigma_A\mathsf{S}$ modulo $\mathsf{G}_{nil}$.

  By \cite[Lemma 6.6]{Jorge:2002}, there are finite subsets $Y_{i,x}$ of $\Omega_A^\kappa\mathsf{S}$ such that $$Cl(Y_{i,x}^{*})\subseteq Cl(R_{i,x}^{*}),$$ and the subgroup generated by $Y_{i,x}$ is equal to the subgroup generated by $R_{i,x}$ in  the free group ($x\in X$ and $1\leq i\leq t_x$). Hence, we have
 \begin{equation}\label{eq26}
   Cl_{\mathsf{G}_{p}}(R_{0,x}^{*}u_{0,x}R_{1,x}^{*}\ldots u_{t_x-1,x}R_{t_x,x}^{*})=Cl_{\mathsf{G}_{p}}(Y_{0,x}^{*}u_{0,x}Y_{1,x}^{*}\ldots u_{t_x-1,x}Y_{t_x,x}^{*}).
 \end{equation}

  Let $S_1(\left<Y_{i,x}\right>)$ be as in Proposition \ref{lemma52} and $S=\displaystyle\bigcap_{x\in X}\displaystyle\bigcap_{i=1}^{t_x}S_1(\left<Y_{i,x}\right>)$. Fix $q\in S$. Since the system \eqref{1} with constraints  $L'_x\subseteq \Omega_{A}\mathsf{S}$ has a solution  modulo $\mathsf{G}_{q}$ and $\mathsf{G}_{q}$ is $\sigma$-reducible, there is a solution $\delta_0:X\rightarrow{{\Omega}}^\sigma_A\mathsf{S}$ modulo $\mathsf{G}_q$. Hence, for every $x\in X$, $\delta_0(x)$ lies in $Cl_{\mathsf{G}_{q},\sigma}(L_x')=Cl_{\mathsf{G}_{q},\kappa}(L_x')$.   For every $p\in S$, we have
  \begin{equation*}
  \begin{aligned}
    Cl_{\mathsf{G}_{p},\kappa}(L_x')&=Cl_{\mathsf{G}_{p},\kappa}\left(R_{0,x}^{*}u_{0,x}R_{1,x}^{*}\ldots u_{t_x-1,x}R_{t_x,x}^{*}\right)\\
    &=Cl_{\mathsf{G}_{p},\kappa}(R_{0,x}^{*})u_{0,x}Cl_{\mathsf{G}_{p},\kappa}(R_{1,x}^{*})\ldots
     u_{t_x-1,x}Cl_{\mathsf{G}_{p},\kappa}(R_{t_x,x}^{*})\\
     &=Cl_{\mathsf{G}_{p},\kappa}(\left<R_{0,x}\right>)u_{0,x}Cl_{\mathsf{G}_{p},\kappa}(\left<R_{1,x}\right>)\ldots
     u_{t_x-1,x}Cl_{\mathsf{G}_{p},\kappa}(\left<R_{t_x,x}\right>)\\
     &=Cl_{\mathsf{G}_{p},\kappa}(\left<Y_{0,x}\right>)u_{0,x}Cl_{\mathsf{G}_{p},\kappa}(\left<Y_{1,x}\right>)\ldots
     u_{t_x-1,x}Cl_{\mathsf{G}_{p},\kappa}(\left<Y_{t_x,x}\right>)\\
      &=Cl_{\mathsf{G}_{q},\kappa}(\left<Y_{0,x}\right>)u_{0,x}Cl_{\mathsf{G}_{q},\kappa}(\left<Y_{1,x}\right>)\ldots
     u_{t_x-1,x}Cl_{\mathsf{G}_{q},\kappa}(\left<Y_{t_x,x}\right>)\\
     &=Cl_{\mathsf{G}_{q},\kappa}\left(Y_{0,x}^{*}u_{0,x}Y_{1,x}^{*}\ldots u_{t_x-1,x}Y_{t_x,x}^{*}\right)\\
     &=Cl_{\mathsf{G}_{q},\kappa}(L_x').
      \end{aligned}
  \end{equation*}
  So, for every $p\in S$, $\delta_0(x)$  lies in  $Cl_{\mathsf{G}_{p},\kappa}(L_x')$.

  By Proposition \ref{lemma52}, there are  $v_{i,x}\in Cl(Y_{i,x}^{*})\cap \Omega_A^{\sigma}\mathsf{S}$ such that if $p\in \mathds{P}\setminus S$, then the pseudoidentity  $v_{i,x}=1$ is valid in $\mathsf{G}_p$ and the pseudoidentity
   \begin{equation}\label{eq17}
     v_{0,x}u_{0,x}v_{1,x}\ldots u_{t_x-1,x}v_{t_x,x}=\delta_0(x).
   \end{equation}
is valid in $\mathsf{G}_p$ otherwise ($x\in X$ and $1\leq i\leq t_x$). Note that since $Cl(Y_{i,x}^{*})\subseteq Cl(R_{i,x}^{*})$, the pseudowords $v_{0,x}u_{0,x}v_{1,x}\ldots u_{t_x-1,x}v_{t_x,x}$ lie in $Cl(L_x')$.

  Let $P=\{p_1,\ldots,p_r\}=\mathds{P}\setminus S$. For every $p_j\in P$, there is a solution  $\delta_j:X\rightarrow{{\Omega}}^\sigma_A\mathsf{S}$ modulo $\mathsf{G}_{p_j}$ such that for every $x\in X$, $\delta_j(x)$ lies in $Cl_{\mathsf{G}_{p_j},\kappa}(L_x')$.

  By Proposition \ref{lemma52}, there are  $v_{i,x,p_j}\in Cl(Y_{i,x}^*)\cap \Omega_A^{\sigma}\mathsf{S}$ such that if $p=p_j$, then the equality
 \begin{equation}\label{eq18}
v_{0,i,p_j}u_{0,x}v_{1,x,p_j}\ldots u_{t_x-1,x}v_{t_x,x,p_j}=\delta_j(x)
  \end{equation}
holds in $\mathsf{G}_p$ and the equality $v_{i,x,p_j}=1$ is valid in $\mathsf{G}_p$ otherwise ($x\in X$, $1\leq i\leq t_x$, $1\leq j\leq r$). Note that since $Cl(Y_{i,x}^*)\subseteq Cl(R_{i,x}^*)$, the pseudowords $$v_{0,x,p_j}u_{0,x}v_{1,x,p_j}\ldots u_{t_x-1,x}v_{t_x,x,p_j}$$ lie in $Cl(L_x')$.

  We define the function $\delta':X\rightarrow{\Omega}_A^{\sigma}\mathsf{S}$ with
  \begin{equation*}
    \begin{aligned}
     \delta'(x)&=\left(v_{0,x}v_{0,x,p_1}\ldots v_{0,x,p_r}\right)u_{0,x}\left(v_{1,x}v_{1,x,p_1}\ldots v_{1,x,p_r}\right)\\
     &u_{t_x-1,x}\left(v_{t_x,x}v_{t_x,x,p_1}\ldots v_{t_x,x,p_r}\right).
    \end{aligned}
  \end{equation*}
  We claim that $\delta'$ is a solution of the system \eqref{1} modulo $\mathsf{G}_{nil}$.

It remains to establish the claim. Note that since a closure of a subsemigroup is again a subsemigroup, the pseudoword $v_{i,x}v_{i,x,p_1}\ldots v_{i,x,p_r}$ lies in $Cl(Y_{i,x}^*)$ and so, $\delta'(x)$ lies in $Cl(L_x')$. We show that the pseudovariety $\mathsf{G}_{nil}$ satisfies the pseudoidentities $\delta'(u_i)=\delta'(v_i)$ ($1\leq i\leq m$). By Lemma \ref{lemma53}, it is enough to show that for every prime number $p$,  the pseudovariety $\mathsf{G}_{p}$ satisfies the pseudoidentities $\delta'(u_i)=\delta'(v_i)$ ($1\leq i\leq m$).

First we show that
 \begin{enumerate}
         \item if $p\in S$, then for every $x\in X$ the pseudoidentity
         \begin{equation}\label{eq19}
            \delta'(x)=\delta_0(x)
         \end{equation}
        is valid in  $\mathsf{G}_p$;
         \item if $p=p_m\in \mathds{P}\setminus S$, then for every $x\in X$ the pseudoidentity
         \begin{equation}\label{eq22}
            \delta'(x)=\delta_m(x)
         \end{equation}
        is valid in  $\mathsf{G}_{p_m}$.
       \end{enumerate}

    We consider the following cases:
\begin{itemize}
  \item  Let $p\in S$. Then for every $x\in X$, the following pseudoidentities hold in $\mathsf{G}_p$:
  \begin{equation*}
    \begin{aligned}
    v_{i,x,p_j}=1&&&&&&(1\leq i\leq t_{x} \ \text{and} \ 1\leq j\leq r), \\
    \end{aligned}
  \end{equation*}
  Hence, the following pseudoidentities hold in $\mathsf{G}_p$:
    \begin{equation*}
    \begin{aligned}
   \delta'(x)&=v_{0,x}u_{0,x}v_{1,x}u_{t_{x}-1,x}v_{t_{x},x}\stackrel{\eqref{eq17}}{=}
   \delta_0(x).
  \end{aligned}
  \end{equation*}
  \item Let $p=p_m\in \mathds{P}\setminus S$.
  Then the following pseudoidentities hold in $\mathsf{G}_p$:
  \begin{equation*}
    \begin{aligned}
    v_{i,x,p_j}=v_{s,x}=1&&&&&&(1\leq i,s\leq t_{x} \ \text{and} \ j\neq m).
    \end{aligned}
  \end{equation*}
  Hence, the following pseudoidentities hold in $\mathsf{G}_{p_m}$:
    \begin{equation*}
    \begin{aligned}
   \delta'(x)&=v_{0,x,p_m}u_{0,x}v_{1,x,p_m}u_{t_{x}-1,x}v_{t_{x},x,p_m}\stackrel{\eqref{eq18}}{=}
   \delta_m(x).
  \end{aligned}
  \end{equation*}
\end{itemize}

Now we show that for every prime number $p$,  the pseudovariety $\mathsf{G}_{p}$ satisfies the pseudoidentities $\delta'(u_i)=\delta'(v_i)$ ($1\leq i\leq m$).
We consider the two following cases:
\begin{itemize}
  \item Let $p\in S$. Since $\delta_0$ is the solution of the system \eqref{1} modulo $\mathsf{G}_p$,  $\delta_0(x)=\delta'(x)$ ($x\in X$), and $u_i,v_i\in \Omega_X^\sigma\mathsf{S}$, the pseudovariety $\mathsf{G}_p$ satisfies the pseudoidentities $\delta'(u_i)=\delta'(v_i)$ ($1\leq i \leq m$).
  \item consider  $p_j\in \mathds{P}\setminus S$. Since $\delta_j$ is the solution of the system \eqref{1} modulo $\mathsf{G}_{p_j}$,  $\delta_j(x)=\delta'(x)$ ($x\in X$), and $u_i,v_i\in \Omega_X^\sigma\mathsf{S}$, the pseudovariety $\mathsf{G}_{p_j}$ satisfies the pseudoidentities $\delta'(u_i)=\delta'(v_i)$ ($1\leq i \leq m$).
\end{itemize}
This proves the theorem.
\end{proof}

\begin{corollary}
  The pseudovariety $\mathsf{G}_{nil}$ is $\sigma$-reducible with respect to the systems of equations associated with finite
directed graphs.
\end{corollary}
\begin{proof}
The result follows from Corollary \ref{cor1} and the preceding theorem.
\end{proof}

\begin{corollary}
   The pseudovariety $\mathsf{G}_{nil}$ is not completely  $\sigma$-reducible.
\end{corollary}
\begin{proof}
  By the preceding Theorem, it is enough to show that for some prime number $p$, the pseudovariety $\mathbf{G}_p$ is not  completely $\sigma$-reducible. Let $p$ be a odd prime number, $A=\{a,b\}$ and consider the following equation
   \begin{equation}\label{eq30}
     [x^2a,y^{-1}z^2by]=1
   \end{equation}
   It has been shown that the equation \eqref{eq30}  does not have solution in the free group \cite{Coulbois&&Khelif}.
We consider the following constraints:
\begin{itemize}
  \item $L_x=\{a\}^*$;
  \item $L_z=\{b\}^*$;
\item $L_y=A^*$.
\end{itemize}
Let $p$ be an odd prime number. We find a solution of the equation \eqref{eq30} with the above  constraints modulo $\mathsf{G}_p$. The pseudowords  $x=(a^{\omega-1})^{2^{\omega-1}}$, $z=(b^{\omega-1})^{2^{\omega-1}}$ and $y=1$ are a solutions of this equation modulo $\mathsf{G}_p$ because  the pseudoidentities  $(a^{\omega-1})^{2^{\omega}}=a^{\omega-1}$ and $(b^{\omega-1})^{2^{\omega}}=b^{\omega-1}$ hold in $\mathsf{G}_p$.

Since $\mathsf{G}_p$ is $\sigma$-full, $\mathsf{G}_p$ is $\sigma$-reducible for the equation  $\eqref{eq30}$ if and only if it has a solution $\delta:\overline{\Omega}_X\mathsf{S}\rightarrow\overline{\Omega}_A \mathsf{S}$ such that $\psi_{\mathsf{G}_p}(\delta(u))=1$ and  $\psi_{\mathsf{G}_p}(\delta(x))\in {\Omega}_A^\sigma \mathsf{G}_p$.
But as $\Omega_A^\sigma \mathsf{G}_p=\Omega_A^\kappa \mathsf{G}_p$,  the equation \eqref{eq30} must have a solution in the free group which is contradiction.
\end{proof}

\section*{Acknowledgments}
This work is part of the author's  preparation of a doctoral thesis under the supervision of Prof. Jorge Almeida, whose advice is gratefully acknowledged.

It was partially supported by the FCT Docoral Grant with reference (SFRH/ BD/98202/2013).
 It was also partially supported by CMUP (UID
 /MAT/00144/ 2013), which is funded by FCT (Portugal) with national (MEC) and European structural funds (FEDER), under the partnership agreement PT2020.

\bibliographystyle{amsplain}

\begin{thebibliography}{10}

\bibitem{Albert&Baldinger&Rhodes:1992}
D.~Albert, R.~Baldinger, and J.~Rhodes, \emph{Undecidability of the identity
  problem for finite semigroups}, J.~Symbolic Logic \textbf{57} (1992),
  179--192.

\bibitem{Jorge:1994}
J.~Almeida, \emph{Finite semigroups and universal algebra}, Series in Algebra,
  vol.~3, World Scientific Publishing Co., Inc., River Edge, NJ, 1994,
  Translated from the 1992 Portuguese original and revised by the author.

\bibitem{Jorge:2002}
\bysame, \emph{Dynamics of implicit operations and tameness of pseudovarieties
  of groups}, Trans. Amer. Math. Soc. \textbf{354} (2002), 387--411.

\bibitem{Almeida&Azevedo&Zeitoun:1997}
J.~Almeida, A.~Azevedo, and M.~Zeitoun, \emph{Pseudovariety joins involving
  {$J$}-trivial semigroups}, Internat.~J.~Algebra Comput. \textbf{9} (1999),
  99--112.

\bibitem{Almeida&Costa&Teixeira:2010}
J.~Almeida, J.~C. Costa, and M.~L. Teixeira, \emph{Semidirect product with an
  order-computable pseudovariety and tameness}, Semigroup Forum \textbf{81}
  (2010), 26--50.

\bibitem{Almeida&Costa&Zeitoun:2004}
J.~Almeida, J.~C. Costa, and M.~Zeitoun, \emph{Tameness of pseudovariety joins
  involving {R}}, Monatsh.~Math. \textbf{146} (2005), 89--111.

\bibitem{Jorge&Zeitoun&Carlos:2007}
\bysame, \emph{Complete reducibility of systems of equations with respect to
  {$R$}}, Port.~Math.~(N.S.) \textbf{64} (2007), 445--508.

\bibitem{Almeida&Delgado:1997}
J.~Almeida and M.~Delgado, \emph{Sur certains syst\`emes d'\'equations avec
  contraintes dans un groupe libre}, Portugal.~Math. \textbf{56} (1999),
  409--417.

\bibitem{Almeida&Delgado:1999}
\bysame, \emph{Sur certains syst\`emes d'\'equations avec contraintes dans un
  groupe libre---addenda}, Portugal.~Math. \textbf{58} (2001), 379--387.

\bibitem{Jorge&&Delgado:2005}
\bysame, \emph{Tameness of the pseudovariety of abelian groups},
  Internat.~J.~Algebra Comput. \textbf{15} (2005), 327--338.

\bibitem{Jorge&Ben:2000}
J.~Almeida and B.~Steinberg, \emph{On the decidability of iterated semidirect
  products with applications to complexity}, Proc.~London~Math.~Soc.~(3)
  \textbf{80} (2000), 50--74.

\bibitem{Ash:1991}
C.~J. Ash, \emph{Inevitable graphs: a proof of the type {II} conjecture and
  some related decision procedures}, Internat.~J.~Algebra Comput. \textbf{1}
  (1991), 127--146.

\bibitem{Coulbois&&Khelif}
T.~Coulbois and A.~Kh\'{e}lif, \emph{Equations in free groups are not finitely
  approximable}, Proc.~Amer.~Math.~Soc. \textbf{127} (1999), 963--965.

\bibitem{Eilenberg:1974}
S.~Eilenberg, \emph{Automata, languages and machines}, vol.~A, Academic Press,
  New York, 1974.

\bibitem{Gitik:1999b}
R.~Gitik, \emph{On the profinite topology on negatively curved groups},
  J.~Algebra \textbf{219} (1999), 80--86.

\bibitem{Gitik&Rips:1995}
R.~Gitik and E.~Rips, \emph{On separability properties of groups},
  Internat.~J.~Algebra Comput. \textbf{5} (1995), 703--717.

\bibitem{Herwig&Lascar:1997}
B.~Herwig and D.~Lascar, \emph{Extending partial automorphisms and the
  profinite topology on free groups}, Trans. Amer. Math. Soc. \textbf{352}
  (2000), 1985--2021.

\bibitem{Margolis&Sapir&Weil:2001}
S.~Margolis, M.~Sapir, and P.~Weil, \emph{Closed subgroups in pro-{$\bold V$}
  topologies and the extension problem for inverse automata},
  Internat.~J.~Algebra Comput. \textbf{11} (2001), 405--445.

\bibitem{Rhodes:1997c}
J.~Rhodes, \emph{Undecidability, automata and pseudovarieties of finite
  semigroups}, Internat.~J.~Algebra Comput. \textbf{9} (1999), 455--473.

\bibitem{Ribes&Zalesskii:1993b}
L.~Ribes and P.~A. Zalesski{\u\i}, \emph{The pro-$p$ topology of a free group
  and algorithmic problems in semigroups}, Internat.~J.~Algebra Comput.
  \textbf{4} (1994), 359--374.

\bibitem{Steinberg:1998a}
B.~Steinberg, \emph{Inevitable graphs and profinite topologies: some solutions
  to algorithmic problems in monoid and automata theory, stemming from group
  theory}, Internat.~J.~Algebra Comput. \textbf{11} (2001), 25--71.

\end{thebibliography}
\providecommand{\bysame}{\leavevmode\hbox to3em{\hrulefill}\thinspace}
\providecommand{\MR}{\relax\ifhmode\unskip\space\fi MR }
\providecommand{\MRhref}[2]{%
  \href{http://www.ams.org/mathscinet-getitem?mr=#1}{#2}
}
\providecommand{\href}[2]{#2}

\end{document}